\documentclass[psamsfonts,11pt,reqno]{amsart}
\usepackage[left=1in,right=1in,top=1.00in,bottom=1.00in]{geometry}
\usepackage{amssymb,mathrsfs}
\newtheorem{theorem}{Theorem}[section]
\newtheorem{lemma}[theorem]{Lemma}
\newtheorem{corollary}[theorem]{Corollary}

\theoremstyle{definition}
\newtheorem{definition}[theorem]{Definition}
\newtheorem{example}[theorem]{Example}
\newtheorem{question}[theorem]{Question}
\newtheorem{problem}[theorem]{Problem}
\newtheorem{remark}[theorem]{Remark}

\theoremstyle{remark}
\newtheorem{claim}{Claim}

\newcommand\N{\mathbb{N}}
\newcommand\cont{\mathfrak{c}}

\title{On the existence of kings in continuous tournaments}

\author[M. Nagao]{Masato Nagao}
\address{Division of Mathematics, Physics and Earth Sciences\\
Graduate School of Science and Engineering\\
Ehime University\\
Matsuyama 790-8577\\
Japan}
\email[M. Nagao]{masa.math@live.jp}
\email[D. Shakhmatov]{dmitri.shakhmatov@ehime-u.ac.jp}

\author[D. Shakhmatov]{Dmitri Shakhmatov}

\thanks{This manuscript was written as part of the first named author's 
doctoral studies at Ehime University.}

\thanks{The second named author was partially supported by the Grant-in-Aid for
Scientific Research (C) No.~22540089 by the Japan Society for the Promotion of Science (JSPS)}



\begin{document}
\begin{abstract}
The classical result of Landau on the existence of kings in finite tournaments
(=~finite 
directed complete
graphs) is extended to continuous tournaments for which the set $X$ of players is a compact Hausdorff space. 
The following partial converse is proved as well.
Let $X$ be a Tychonoff space which is either
zero-dimensional or locally connected or pseudocompact or linearly ordered. If $X$ admits at least one continuous tournament and each continuous tournament on $X$ has a king,
then $X$ must be compact. We show that a complete reversal of our theorem is impossible, by giving an example of a 
dense connected subspace $Y$ of the unit square admitting precisely two continuous tournaments both of which have a king, yet $Y$ is not 
even analytic
(much less compact).
\end{abstract}

\maketitle

\section{The classical ``king chicken'' theorem of Landau}

For a set $X$, we use $[X]^2$ to denote the set of all two-element subsets of $X$. 
A {\em weak selection\/} on $X$ is a function $\varphi:[X]^2\to X$ such that $\varphi(\{a,b\})\in \{a,b\}$ for all $\{a,b\}\in [X]^2$.
Clearly, $\varphi$ defines a 
tournament (= a directed complete
graph) on the set $X$
in which a team $a\in X$ 
wins over a team $b\in X$ if and only if
$\varphi(\{a,b\})=b$.

In our terminology, the 
classical ``king chicken'' theorem of Landau (\cite{Landau}; see also \cite{Maurer}) 
about the existence of kings in 
finite tournaments (= finite 
directed complete
graphs)  
reads as follows.

\begin{theorem}
\label{classical:king:theorem}
Let $\varphi$ be a weak selection on a 
non-empty
finite set $X$.
Then there exists $z\in X$ such that, for every $x\in X\setminus\{z\}$, 
either 
$\varphi(\{x,z\})=x$
or one can find 
an element
$y\in X$
different from both $x$ and $z$
such that 
$\varphi(\{x,y\})=x$
and
$\varphi(\{y,z\})=y$.
\end{theorem}

An element $z$ in the above theorem is called a {\em king\/} for the tournament $\varphi$.

We
extend this classical theorem from the finite case to all compact 
Hausdorff spaces $X$ equipped with a {\em continuous\/} weak selection $\varphi$;
see Theorem \ref{kings:in:compact:spaces}.
We also show that the compactness of $X$ is not only sufficient but often also a necessary condition for the existence of a king for every continuous tournament $\varphi$ on $X$; see Corollary \ref{linearly:ordered:king} and Theorem \ref{reverse:theorem}.
Finally, we give examples of non-compact separable metric spaces that nevertheless 
have kings for all continuous tournaments $\varphi$ on $X$; see 
Examples  \ref{example:2} and \ref{example}. The first example is complete, while 
the second example fails to be even analytic.

\section{Compact version of the ``king chicken'' theorem}

\begin{definition}
\label{def:main}
Let $\varphi$ be a weak selection on a set $X$.
\begin{itemize}
\item[(i)]
For $a,b\in X$, we write $a\to_\varphi b$ if either $a=b$ or $a\not=b$ and $\varphi(\{a,b\})=b$.
\item[(ii)]
We call an element $z\in X$ a {\em $\varphi$-king\/} if, 
for every $x\in X$, there exists $y\in X$ such that
$z\to_\varphi y\to_\varphi x$. 
\item[(iii)]
For every $x\in X$, define
\begin{equation}
\label{eq:K}
K_{\varphi,x}=\{z\in X: z\to_\varphi y\to_\varphi x
\mbox{ for some }
y\in X\}.
\end{equation}
Clearly, 
\begin{equation}
\label{eq:K:phi}
K_\varphi=\bigcap\{K_{\varphi,x}: x\in X\}
\end{equation}
is the set of all $\varphi$-kings.
\end{itemize}
\end{definition}

\begin{definition}
Let $X$ be a topological space.
\begin{itemize}
\item[(i)]
A weak selection $\varphi$ on $X$ is {\em continuous\/} provided that,
for every $a,b\in X$ with $a\not=b$ and each open neighbourhood 
$W$ of $\varphi(\{a,b\})$,
there exist an open neighbourhood $U$ of $a$ 
and an open neigbourhood $V$ of $b$ such that $\varphi(\{a',b'\})\in W$
whenever $a'\in U$, $b'\in V$ and $a'\not=b'$.
\item[(ii)]
We say that 
$X$ is a {\em king space\/} if
every continuous weak selection $\varphi$ on $X$ has a $\varphi$-king; or equivalently,
if $K_\varphi\not=\emptyset$ for every continuous 
weak selection $\varphi$ on $X$.
\end{itemize}
\end{definition}

Using our terminology, 
Theorem \ref{classical:king:theorem}
can be 
restated as follows: 
{\em Every 
non-empty
finite discrete space is a king space\/}. 
Our first theorem 
extends 
this result 
to all compact Hausdorff spaces.

\begin{theorem}
\label{kings:in:compact:spaces}
Every 
non-empty
compact Hausdorff space is a king space.
\end{theorem}
\begin{proof}
Let $X$ be a 
non-empty
compact Hausdorff space, and let $\varphi$ be a continuous 
weak selection on $X$.
We must prove that $K_\varphi\not=\emptyset$.

\begin{claim}
\label{K:is:closed}
$K_{\varphi,x}$ is closed in $X$ for every $x\in X$.
\end{claim} 
\begin{proof}
Since $X$ is Hausdorff and
$\varphi$ is continuous, 
$F=\{(a,b)\in X^2: a\to_\varphi b\}$
is a closed subset of $X^2$.
Then $F\times X$ and $X\times F$ are closed subsets of $X^3$, and so is
their intersection
\begin{equation}
\label{eq:Phi}
P=(F\times X)\cap (X\times F)
=
\{(a,b,c)\in X^3: a\to_\varphi b\to_\varphi c\}.
\end{equation}

Let $x\in X$. Since 
the set $X\times X\times \{x\}$ is closed in $X^3$,
the set  
\begin{equation}
\label{eq:P}
Q=P\cap (X\times X\times \{x\})
\end{equation}
 is also closed in $X^3$.
Since $X^3$ is compact, 
$Q$ is compact as well.
Thus,
$\pi(Q)$ is compact too, where $\pi:X^3\to X$ is the (continuous) 
projection on the first coordinate.
Since $X$ is Hausdorff, $\pi(Q)$ must be closed in $X$. It remains only to note that
$$
\pi(Q)=\{z\in X: \exists\ y\in X\ (z\to_\varphi y\to_\varphi x)\}=K_{\varphi,x}
$$
by \eqref{eq:K}, \eqref{eq:Phi} and \eqref{eq:P}.
\end{proof}
\begin{claim}
\label{fip}
The family $\{K_{\varphi,x}:x\in X\}$ has the finite intersection property.
\end{claim}
\begin{proof}
Let $S$ be 
a non-empty
finite subset of $X$.
The restriction $\psi$  of $\varphi$ to $[S]^2$
is a 
weak selection on $S$. 
By 
Theorem \ref{classical:king:theorem},
there exists a $\psi$-king; 
that is, $\bigcap\{K_{\psi,x}:x\in S\}\not=\emptyset$.
Note that
$K_{\psi,x}\subseteq K_{\varphi,x}$ for every $x\in S$, so 
$\bigcap\{K_{\varphi,x}:x\in S\}\not=\emptyset$
as well.
\end{proof}
Since $X$ is compact, from Claims \ref{K:is:closed}
and \ref{fip} we conclude that 
$K_\varphi=\bigcap\{K_{\varphi,x}:x\in X\}\not=\emptyset$.
\end{proof}

\begin{remark}
The proof of Claim \ref{K:is:closed}
is a straightforward adaptation to our needs of the proof of Theorem 2.2 in \cite{McGehee} which asserts that the composition of two closed relations on a compact Hausdorff space $X$ is a 
closed relation on $X$. The latter result itself 
is a corollary of a more general
Theorem 2.6 in \cite{Magill} which states that the composition of two compact relations on a Hausdorff space $X$ is a compact relation on $X$.
\end{remark}

\section{King spaces are often compact}

In this section we shall obtain a partial converse to Theorem \ref{kings:in:compact:spaces}; see Corollary \ref{linearly:ordered:king} and Theorem \ref{reverse:theorem}. In order to do this, 
we first establish some general properties of king spaces.

\begin{lemma}
\label{disjoint:sum}
Let $X$ be a king space having a continuous weak selection.
Then every 
non-empty
clopen subset of $X$ is a king space as well.
\end{lemma}
\begin{proof}
Fix a continuous weak selection $\psi$ on $X$.
Let $U$ be a 
non-empty
clopen subset of $X$.
Consider 
an arbitrary 
continuous weak 
selection 
$\xi$
on $U$. 
Our goal is to find a $\xi$-king.
Define the map $\varphi:[X]^2\to X$ by 
\begin{equation}
\label{eq:disjoint:sum}
\varphi(\{a,b\})=
\left\{ \begin{array}{ll}
\xi(\{a,b\}), & \mbox{ if } a,b\in U\\
\psi(\{a,b\}), & \mbox{ if } a,b\in X\setminus U\\
b, & \mbox{ if } a\in U \mbox{ and } b\in X\setminus U
\end{array} \right.
\hskip10pt
\mbox{for $\{a,b\}\in[X]^2$.}
\end{equation}
Since $U$ and $X\setminus U$ are open in $X$, and both 
maps $\xi$ and $\psi$ are continuous, it easily follows that 
$\varphi$ is a continuous weak selection on $X$.
Since $X$ is a king space, there exists
a $\varphi$-king
$z\in X$. 
\begin{claim}
\label{new:claim}
If $a\in U$, $b\in X$ and $b\to_\varphi a$, then $b\in U$ and $b\to_\xi a$. 
\end{claim}
\begin{proof}
Since $a\to_\xi a$ holds, without loss of generality, we shall assume that 
$b\not=a$. From $b\to_\varphi a$ and Definition \ref{def:main}(i), we get $\varphi(\{a,b\})=a$,
and so $b\in U$ by \eqref{eq:disjoint:sum}.
Since $a,b\in U$, applying \eqref{eq:disjoint:sum} once again,
we obtain $a=\varphi(\{a,b\})=\xi(\{a,b\})$,
which implies $b\to_\xi a$. 
\end{proof}
Let $x\in U$ be arbitrary. (Note that at least one such $x$ exists, as $U\not=\emptyset$.)
Since $z$ is a $\varphi$-king,
$z\to_\varphi y\to_\varphi x$ for some $y\in X$.
Applying Claim \ref{new:claim} twice,
we consequently get $y,z\in U$ and $z\to_\xi y\to_\xi x$.
This proves that $z\in U$ is a $\xi$-king.
\end{proof}

Recall that a {\em linearly ordered space\/} is a topological 
space $X$ equipped with a linear order $<$ such that the family 
$\{\{x\in X: x<p\}:p\in X\}\cup \{\{x\in X: x>p\}:p\in X\}$
is a subbase for the topology of $X$.
This order $<$ is said to {\em generate\/} the topology of $X$.  

\begin{lemma}
\label{reverse:lemma}
No 
proper dense subspace of a linearly ordered space is a king space.
\end{lemma}
\begin{proof}
Let $X$ be a 
proper dense subspace of a linearly ordered space $Y$.
Clearly, $X\not=\emptyset$.
Let $<$ be the linear order on $Y$ generating the topology of $Y$.
Choose $p\in Y\setminus X$ arbitrarily, 
and 
note that 
\begin{equation}
\label{two:rays}
X^{\leftarrow}=\{x\in X: x<p\}
\ \ \mbox{ and }\ \ 
X^{\to}=\{x\in X: p<x\}
\end{equation}
are open subsets of $X$ such that
$X^{\leftarrow}\cap X^{\to}=\emptyset$ and $X=X^{\leftarrow}\cup X^{\to}$.
In particular, both $X^{\leftarrow}$ and $X^{\to}$ are clopen in $X$.
Now
we shall consider 
four
cases.

\smallskip
{\em Case 1\/}. 
{\sl $X^{\leftarrow}\not=\emptyset$ and the ordered set $(X^{\leftarrow},<)$ has no maximal element\/}.
Since $Y$ is a linearly ordered space, it has 
a
continuous weak selection
$\eta$
defined by $\eta(\{a,b\})=\min\{a,b\}$
for $\{a,b\}\in [Y]^2$.
Consider the continuous weak selection
$\mu=\eta\restriction_{[X^{\leftarrow}]^2}$
on 
$X^{\leftarrow}$.
Suppose that $z\in X^{\leftarrow}$ is a $\mu$-king.
Since $(X^{\leftarrow},<)$ has no maximal element,
$z<x$ for some $x\in X^{\leftarrow}$. 
Since $z$ is a $\mu$-king, there exists $y\in X^{\leftarrow}$
with 
$z\to_\mu y\to_\mu x$.
Recalling the definition of $\mu$, we get
$z\ge y\ge x$, in contradiction with $z<x$.
This shows that 
there is no $\mu$-king,
and hence, $X^{\leftarrow}$ 
is not a king space.
Since $X^{\leftarrow}$ is a 
non-empty
clopen subset of 
$X$ and $X$ has the continuous weak selection
$\eta\restriction_{[X]^2}$, 
from Lemma \ref{disjoint:sum}
we conclude that $X$ is not a king space either.

\smallskip
{\em Case 2\/}. 
{\sl $X^{\to}\not=\emptyset$ and the ordered set $(X^{\to},<)$ has no minimal element\/}.
By considering the reverse order on $Y$, one easily reduces this case to 
Case 1.

\smallskip
{\em Case 3\/}. 
{\sl $X^{\leftarrow}=\emptyset$\/}.
In this case $X^{\rightarrow}=X\not=\emptyset$.
If 
$q$
is a minimal element of $(X^{\to},<)$, then 
$V=\{y\in Y: y<q\}$ 
is an open subset of $Y$
such that $p\in V\not=\emptyset$.
Since $X$ is dense in $Y$,
$V\cap X\not=\emptyset$.
Since $V\cap X^{\to}=\emptyset$
by the minimality of $q$,
it follows that
$\emptyset\not=V\cap X= V\cap X^{\leftarrow}\subseteq X^{\leftarrow}$, 
in contradiction
with $X^{\leftarrow}=\emptyset$.
Therefore, $(X^{\to},<)$ does not have a minimal element. That is, Case 3 is reduced to Case 2.

\smallskip
{\em Case 4\/}. 
{\sl $X^{\to}=\emptyset$\/}.
By considering the reverse order on $Y$, one easily reduces this case to 
Case 3.

\smallskip
To finish the proof, it remains to 
observe that the four cases above exhaust all possibilities.
Indeed, suppose that none of Cases  1--4 holds.
Then both $x_1=\max X^{\leftarrow}$ and $x_2=\min X^{\to}$ exist. 
From this and \eqref{two:rays}, we conclude that
$W=\{z\in Y: x_1<z<x_2\}$ is an open subset of $Y$
such that $p\in W\not=\emptyset$ and $W\cap X=\emptyset$.
This contradicts the density of $X$ in $Y$.
\end{proof}

The next corollary provides our first 
partial converse to Theorem \ref{kings:in:compact:spaces}.

\begin{corollary}
\label{linearly:ordered:king}
A linearly ordered king space is compact.
\end{corollary}
\begin{proof}
Let $X$ be a 
linearly ordered 
king 
space and let 
$<$ be the order on $X$ generating its topology.
Let $(X^+,\prec)$ be the Dedekind compactification of $(X,<)$
obtained by 
``filling all gaps'' in $X$; see \cite[Section 6]{GH}
or \cite[3.12.3(b)]{Engelking}.
Then $Y=(X^+,\prec)$ is a compact linearly ordered space 
containing $X$ as its dense subspace.
Now $X=Y$ by Lemma \ref{reverse:lemma}, so $X$ is compact. 
\end{proof}

\begin{lemma}
\label{countable:disjoint:sum}
Let $X$ be a space having a continuous weak selection.
If $X$ admits a partition $\mathscr{U}=\{U_n:n\in\N\}$ into 
pairwise disjoint non-empty open sets $U_n$, then $X$ is not a king space.
\end{lemma}
\begin{proof}
Fix a continuous weak selection 
$\psi$ on $X$.
Since $\mathscr{U}$ is a partition of $X$,
for every $x\in X$ there exists a unique $n(x)\in\N$ such that
$x\in U_{n(x)}$.
Define the map $\varphi:[X]^2\to X$ as follows:
\begin{equation}
\label{defining:varphi}
\varphi(\{a,b\})=
\left\{ \begin{array}{ll}
a, & \mbox{ if } n(a)<n(b)  \\
b, & \mbox{ if } n(a)>n(b)  \\
\psi(\{a,b\}), & \mbox{ if } n(a)=n(b)
\end{array} \right.
\hskip40pt
\mbox{for $\{a,b\}\in[X]^2$.}
\end{equation}
Using 
the continuity of $\psi$ and 
the fact that $\mathscr{U}$ is a clopen partition of $X$,
one can easily check that $\varphi$ is a continuous 
weak selection on $X$.

Suppose that some $z\in X$ is a $\varphi$-king.
Let $m=n(z)$. Since $U_{m+1}\not=\emptyset$, we can choose 
$x\in U_{m+1}$. Since $z$ is a $\varphi$-king,
$z\to_\varphi y\to_\varphi x$ for some $y\in X$.
From this and \eqref{defining:varphi}, we obtain
$m=n(z)\ge n(y)\ge n(x)=m+1$.
This contradiction shows that there is no
$\varphi$-king, and so 
$X$ is not a king space.
\end{proof}

Recall that a space $X$ is:
\begin{itemize}
\item {\em pseudocompact\/} if every real-valued continuous function defined on $X$ is bounded;
\item
{\em zero-dimensional\/} if $X$ has a base consisting of clopen subsets
of $X$;
\item
{\em locally connected\/} provided that, for every open subset $U$ of $X$
and each point $x\in U$, there exist an open subset $V$ of $X$ and a connected 
subset $C$ of $X$ with $x\in V\subseteq C\subseteq U$;
\item
{\em orderable\/} if there exists a linear order $<$ on $X$ turning it into a linearly ordered space.
\end{itemize}

Now we can present a 
second
partial converse of 
Theorem \ref{kings:in:compact:spaces}.

\begin{theorem}
\label{reverse:theorem}
Let $X$ be a 
Tychonoff 
king space having at least one continuous weak selection. Then $X$ is compact
in each of the following 
cases:
\begin{itemize}
\item[(i)] $X$ is pseudocompact;
\item[(ii)] $X$ is zero-dimensional;
\item[(iii)] $X$ is locally connected.
\end{itemize}
\end{theorem}
\begin{proof}
(i)
Without loss of generality, we may assume that $|X|\ge 2$.
By
\cite[Theorem 2.3]{GF-Sanchis}, $X\times X$ is also pseudocompact,
so 
the Stone-\v{C}ech compactification $\beta X$ of $X$ is orderable
by the main result of \cite{Miya}; see also \cite[Theorem 1.16]{Artico:etc}.
Applying Lemma \ref{reverse:lemma}, we conclude that $X=\beta X$. 
In particular, $X$ is compact.

\smallskip
(ii)
Suppose that $\{V_n:n\in\N\}$ is a discrete family of non-empty open subsets of $X$.
Since $X$ is zero-dimensional, for every $n\in\N$ we can choose a non-empty clopen 
subset $U_n$ of $X$ such that $U_n\subseteq V_n$. Clearly, $\{U_n:n\in\N\}$ is also a discrete family in $X$. Since each $U_n$ is a clopen subset of $X$,
the set $W=\bigcup\{U_n:n\in\N\setminus\{0\}\}$ 
is 
clopen in $X$, and so 
is the set $X\setminus W$.
Replacing $U_0$ with the bigger set $X\setminus W$, we obtain a partition
$\mathscr{U}=\{U_n:n\in\N\}$ of $X$ into pairwise disjoint non-empty
clopen subsets of $X$.
Applying Lemma \ref{countable:disjoint:sum}, we conclude that $X$ is a not a king space, in contradiction with our assumption on $X$.
This contradiction shows that every discrete family of non-empty open subsets of $X$ must be finite. 
Since $X$ is Tychonoff, this implies that $X$ is pseudocompact. 
Applying item 
(i),
we conclude that $X$ is compact.

\smallskip
(iii)
Since $X$ is locally connected, all connected components of $X$ are clopen 
in $X$. If $X$ has infinitely many connected components, then by grouping some of them together, if necessary, we can produce a partition $\mathscr{U}=\{U_n:n\in\N\}$ of $X$ into 
pairwise disjoint non-empty open sets $U_n$, in contradiction with Lemma 
\ref{countable:disjoint:sum}.
This shows that $X$ has only finitely many connected components.
Therefore, in order to establish compactness of $X$, it suffices to show
that each connected component $C$ of $X$ is compact. 
By our assumption, $X$ admits a continuous weak selection $\varphi$, so
$C$ is a connected, locally connected space having a continuous 
weak selection 
$\varphi\restriction_{[C^2]}$.
By \cite[Lemma 11]{Nogura:Shakhmatov}, $C$ is orderable.\footnote{Indeed, \cite[Lemma 10(i)]{Nogura:Shakhmatov} coincides with \cite[Lemma 7.2]{Michael}, so Lemma 10(i) in \cite{Nogura:Shakhmatov}
remains valid for continuous weak selections. Since 
the proof of Lemma 11 in \cite{Nogura:Shakhmatov}
uses only the linear order obtained in \cite[Lemma 10(i)]{Nogura:Shakhmatov},
it follows that \cite[Lemma 11]{Nogura:Shakhmatov} also remains valid for continuous weak selections.}
Since $C$ is a non-empty clopen 
subset of the king space $X$ having the continuous weak selection $\varphi$, 
Lemma \ref{disjoint:sum} implies that $C$ is a king space.
Finally, $C$ must be compact by Corollary \ref{linearly:ordered:king}.
\end{proof}

\begin{remark}
A space without a continuous weak selection is trivially a king space. 
Therefore, one cannot expect to obtain a partial converse of Theorem \ref{kings:in:compact:spaces} in the spirit of Theorem \ref{reverse:theorem}
without the assumption that the space $X$ in question has at least one continuous weak selection. 
\end{remark}

\section{King spaces need not be analytic}

One may be tempted to conjecture that the additional assumptions (i)--(iii)
in Theorem \ref{reverse:theorem} can be omitted, thereby providing the full converse 
of Theorem \ref{kings:in:compact:spaces}.
This section is devoted to showing that this is impossible.
Non-compact 
separable metric 
king spaces with exactly two continuous 
weak selections will be constructed in Examples  \ref{example:2} and \ref{example}. 
The former example is complete, while 
the latter example fails to be even analytic.

We shall need some lemmas for the construction of our examples. 
The first lemma provides a simple method of building king spaces.

Let $I=[0,1]$ and let $f:I\to I$
be a function.
Define
\begin{equation}
\label{eq:x:f}
x^f_s=(s,f(s))\in I^2
\ \mbox{ for each }\ 
s\in I.
\end{equation}
Clearly,
\begin{equation}
\label{eq:X:f}
X^f=\{x^f_s: s\in I\}\subseteq I^2
\end{equation}
is the graph of $f$.
We consider $X^f$
with the subspace topology inherited from  $I^2$,
where the latter space is equipped with the usual topology. 

\begin{lemma}
\label{connected:king}
If $f:I\to I$ is a function such that 
$X^f$ is connected, then
$X^f$ is 
a
king space that has precisely two continuous 
weak
selections.
\end{lemma}
\begin{proof}
Clearly, $X^f$ has two continuous weak selections $\sigma_{\min}$ and $\sigma_{\max}$ 
defined by 
$$
\sigma_{\min}(\{x^f_s, x^f_t\})
=
x^f_{\min\{s,t\}}
\ \ \mbox{ and }\ \ 
\sigma_{\max}(\{x^f_s, x^f_t\})
=
x^f_{\max\{s,t\}}
$$
for $s,t\in I$, $s\not=t$, respectively.
Since $X^f$ is connected, it follows from \cite[Lemma 7.2]{Michael} that there are no other continuous 
weak selections on $X^f$.
Since 
$x^f_1$ is a (unique) $\sigma_{\min}$-king
and
$x^f_0$ is a (unique) $\sigma_{\max}$-king,
we conclude that $X^f$ is a king space.
\end{proof}

Our first example demonstrates that Theorem \ref{reverse:theorem} fails if one replaces ``locally connected'' with ``connected'' in its item (iii).

\begin{example}
\label{example:2}
Let $f:I\to I$ be the function defined by 
$f(0)=0$ and $f(t)=|\sin(1/t)|$ for $0<t\le 1$.
Then {\em
$X^f$ is a non-compact connected completely metrizable separable king space  
that has precisely two continuous 
weak
selections.\/}
Indeed, $X^f$ is a non-closed connected $G_\delta$-subspace of $I^2$.
The rest follows from Lemma \ref{connected:king}.
\end{example}

Complete metrizability of $X^f$ in this example cannot be strengthened to its local compactness, as the next remark shows.

\begin{remark}
\label{locally:compact:remark}
{\em A locally compact connected king space $X$ having at least one 
continuous weak 
selection is compact\/}.
Indeed,
$X$ has a weaker topology generated by a linear order
\cite[Lemma 7.2]{Michael}. 
Now 
$X$ is orderable 
by \cite[Proposition 1.18]{Artico:etc},
and Corollary 
\ref{linearly:ordered:king} implies that $X$ is compact. 
\end{remark}

Our next lemma offers a technique for building connected spaces of the form $X^f$
for some function $f:I\to I$.
Let $\pi:I^2\to I$ be the projection onto the first coordinate.
We use $\cont$ to denote the cardinality of the continuum.
\begin{lemma}
\label{counter-example:2}
Suppose that $f:I\to I$ is a function 
satisfying the following conditions:
\begin{itemize}
\item[(a)]
$X^f$ is dense in $I^2$;
\item[(b)]
$X^f\cap  F\not=\emptyset$ for every closed subset $F$ of $I^2$ such that $\pi(F)$ has cardinality $\cont$.
\end{itemize}
Then 
$X^f$ is connected.
\end{lemma}
\begin{proof}
Let $U'$ and $V'$ be disjoint non-empty open subsets of $X^f$. It suffices to show that 
\begin{equation}
\label{connected:equation}
X^f\setminus(U'\cup V')\not=\emptyset.
\end{equation}
Fix open subsets $U$ and $V$ of $I^2$ such that $U\cap X^f=U'$
and $V\cap X^f=V'$.
Since $X^f$ is dense in $I^2$ by (a),
$U'\cap V'=\emptyset$ yields
$U\cap V=\emptyset$.
Since $U$ and $V$ are non-empty open subsets of $I^2$,
$\pi(U)$ and $\pi(V)$ are non-empty open subsets of $I$.
In particular, the set
$$
W=\pi(U)\cap \pi(V)
$$
is open in $I$.

\begin{claim}
\label{claim:3}
$W\subseteq \pi(F)$, where $F= I^2\setminus (U\cup V)$.
\end{claim}
\begin{proof}
Let $s\in W$. Then $U''=(\{s\}\times I)\cap U$ 
and $V''=(\{s\}\times I)\cap V$ are  disjoint non-empty open subsets 
of the connected space $\{s\}\times I$, which yields
$(\{s\}\times I)\setminus (U\cup V)=(\{s\}\times I)\setminus (U''\cup V'')\not=\emptyset$.
Therefore, $(s,t)\not\in U\cup V$ for some $t\in I$. 
Thus,
$(s,t)\in F$, and so $s=\pi(s,t)\in\pi(F)$. 
\end{proof}

\smallskip
{\em Case 1\/}. 
$W\not=\emptyset$.
Since $W$ is a non-empty open subset of $I$, 
it has cardinality $\cont$. 
From Claim \ref{claim:3} we conclude that $\cont=|I|\ge |\pi(F)|\ge |W|=\cont$;
that is, $|\pi(F)|=\cont$.
Since 
$F$
is closed in $I^2$,
$X^f\cap  F\not=\emptyset$ by 
(b).
Since $X^f\cap  F=X^f\setminus (U\cup V)=X^f\setminus(U'\cup V')$,
this proves \eqref{connected:equation}.

\smallskip
{\em Case 2\/}.
$W=\emptyset$.
In this case $\pi(U)$ and $\pi(V)$ are disjoint non-empty open subsets
 of $I$, 
so
$I\not=\pi(U)\cup \pi(V)$
by the connectedness of $I$.
Choose 
$s\in I\setminus (\pi(U)\cup \pi(V))$.
Then $(\{s\}\times I)\cap  (U\cup V)=\emptyset$.
Since $x^f_s\in \{s\}\times I$
by \eqref{eq:x:f},
it follows that 
$x^f_s\not\in U\cup V$.
Since $U'\subseteq U$ and $V'\subseteq V$, this gives 
$x^f_s\not\in U'\cup V'$ as well.
Since $x^f_s\in X^f$
by \eqref{eq:X:f},
this
proves \eqref{connected:equation}.
\end{proof}

\begin{lemma}
\label{counter-example:3}
For each family $\mathscr{G}$ of subsets of $I^2$ with $|\mathscr{G}|\le\cont$,
there exists a function $f:I\to I$ satisfying 
conditions (a) and (b) of 
Lemma \ref{counter-example:2}
such that 
$X^f\not\in\mathscr{G}$.
\end{lemma}
\begin{proof}
Let $\mathscr{V}=\{V_n:n\in\N\}$ be a countable base for $I^2$ such that all
$V_n$ are non-empty.
By induction on $n\in\N$, we can choose $s_n,t_n\in I$ so that
\begin{equation}
\label{s_n:are:pairwise:distinct}
s_n\not= s_m
\ \mbox{ for }\ 
m,n\in\N
\mbox{ with }
m\not=n
\end{equation}
and
\begin{equation}
\label{ii_n}
(s_n,t_n)\in V_n
\ \mbox{ for all }\ 
n\in\N.
\end{equation}

Since the family $\mathscr{F}$ of all closed subsets of $I^2$ has cardinality $\cont$, we can fix an enumeration 
$\mathscr{F}=\{F_\alpha:\alpha$ is a limit ordinal with $\omega\le \alpha<\cont\}$
of $\mathscr{F}$ such that $F_\omega$ is a singleton. 
Since
$|\mathscr{G}|\le\cont$,
 we can fix an enumeration 
$\mathscr{G}=\{G_\alpha:
\alpha$ is a successor ordinal with $\omega\le \alpha<\cont\}$
of $\mathscr{G}$.

Using transfinite induction, for each ordinal $\alpha$
with $\omega\le \alpha<\cont$
we shall choose $s_\alpha,t_\alpha\in I$ satisfying the following conditions:
\begin{itemize}
\item[(i$_\alpha$)]
$s_\alpha\not\in \{s_\beta: \beta<\alpha\}$;
\item[(ii$_\alpha$)]
if $\alpha$ is a limit ordinal and $|\pi(F_\alpha)|=\cont$, then $(s_\alpha,t_\alpha)\in F_\alpha$;
\item[(iii$_\alpha$)]
if $\alpha$ is a successor ordinal and 
$|(\{s_\alpha\}\times I)\cap G_\alpha|<\cont$, then
$(s_\alpha,t_\alpha)\not\in G_\alpha$.
\end{itemize}

To start the induction, choose $s_\omega\in I\setminus \{s_n:n\in\N\}$
arbitrarily and let $t_\omega=0$.
Then (i$_\omega$) and (iii$_\omega$) trivially hold.
Now $|F_\omega|=1$ implies $|\pi(F_\omega)|=1<\cont$, so 
(ii$_\omega$) trivially holds as well.

Suppose now that $\omega<\alpha<\cont$ and $s_\beta,t_\beta\in I$
satisfying conditions (i$_\beta$)--(iii$_\beta$) have already been 
chosen for each ordinal $\beta$ such that $\omega\le \beta<\alpha$.
We shall choose $s_\alpha,t_\alpha\in I$ satisfying (i$_\alpha$)--(iii$_\alpha$).
We consider two cases.

\smallskip
{\em Case 1\/}. 
{\sl $\alpha$ is a limit ordinal\/}.
Let $F=F_\alpha$ if $|\pi(F_\alpha)|=\cont$, and let $F=I^2$ otherwise. 
Then 
$|\{s_\beta:\beta<\alpha\}|\le |\alpha|<\cont=|\pi(F)|$, 
so we can choose $s_\alpha\in \pi(F)\setminus \{s_\beta:\beta<\alpha\}$.
Since $s_\alpha\in \pi(F)$, there exists $t_\alpha\in I$
with $(s_\alpha,t_\alpha)\in F$.
Then 
(i$_\alpha$) and (ii$_\alpha$) are satisfied. The condition (iii$_\alpha$) is vacuous.

\smallskip
{\em Case 2\/}. 
{\sl $\alpha$ is a successor ordinal\/}.
 Since
$|\{s_\beta:\beta<\alpha\}|\le |\alpha|<\cont=|I|$,
we can choose 
$s_\alpha\in I$ satisfying (i$_\alpha$).
If $|(\{s_\alpha\}\times I)\cap G_\alpha|<\cont$, we can 
choose $t_\alpha\in I$ so that $(s_\alpha,t_\alpha)\not \in G_\alpha$.
 Otherwise, we 
choose $t_\alpha\in I$ arbitrarily. By our choice, (iii$_\alpha$) holds.
The condition (ii$_\alpha$) is vacuous.

\smallskip
The inductive construction has been completed.
Now we define $S=\{s_\alpha:\alpha<\cont\}$.
From \eqref{s_n:are:pairwise:distinct} and the fact that (i$_\alpha$) 
holds for every ordinal $\alpha$ with $\omega\le\alpha<\cont$,
we conclude that 
$s_\beta\not= s_\alpha$ whenever $\beta<\alpha<\cont$.
We claim that the function $f:I\to I$ defined by
$$
f(s)=
\left\{ \begin{array}{ll}
0, & \mbox{ if } s\in I\setminus S \\
t_\alpha, & \mbox{ if } s=s_\alpha \mbox{ for some } \alpha<\cont
\end{array} \right.
\hskip40pt
\mbox{for $s\in I$,}
$$
has the required properties.
From our definition of $f$, \eqref{eq:x:f} and \eqref{eq:X:f}, it follows that
\begin{equation}
\label{point:x}
x_{s_\alpha}^f=(s_\alpha,t_\alpha)\in X^f
\ \mbox{ for every ordinal }\ 
\alpha<\cont.
\end{equation}

First, let us
check conditions (a) and (b) of 
Lemma \ref{counter-example:2}.

\smallskip
(a) 
Let $U$ be a non-empty open subset of $I^2$. Since $\mathscr{V}$ is a base of $I^2$, there exists $n\in\N$ such that $V_n\subseteq U$.
From this, \eqref{ii_n} and  
\eqref{point:x}, 
we get
$(s_n,t_n)\in X^f\cap V_n\subseteq X^f\cap U$,
so
$X^f\cap U\not=\emptyset$. This shows that $X^f$ is dense in $I^2$.

\smallskip
(b) 
Let $F$ be a closed subset of $I^2$ such that $|\pi(F)|=\cont$.
Then 
$F\in\mathscr{F}$,
and so $F=F_\alpha$ for some limit ordinal $\alpha$ with $\omega\le\alpha<\cont$.
From \eqref{point:x} and  (ii$_\alpha$), we conclude that
$(s_\alpha,t_\alpha)\in X^f\cap F_\alpha=X^f\cap F$.
This shows that
$X^f\cap  F\not=\emptyset$.

\smallskip
Second, 
suppose that $X^f\in \mathscr{G}$. Then $X^f=G_\alpha$
for some successor ordinal $\alpha$ with $\omega\le\alpha<\cont$.
Note that 
$|(\{s_\alpha\}\times I)\cap G_\alpha|=|(\{s_\alpha\}\times I)\cap X^f|=|\{x^f_{s_\alpha}\}|=1<\cont$ by \eqref{eq:x:f} and \eqref{eq:X:f},
so
$(s_\alpha,t_\alpha)\not\in G_\alpha$
by 
(iii$_\alpha$).
On the other hand, 
$(s_\alpha,t_\alpha)\in X^f$
by \eqref{point:x}.
This shows that $X^f\not=G_\alpha$,
giving a contradiction with $X^f=G_\alpha$.
This contradiction shows that
$X^f\not\in\mathscr{G}$.
\end{proof}

\begin{corollary}
\label{analytic:corollary}
For each family $\mathscr{G}$ of subsets of $I^2$ with $|\mathscr{G}|\le\cont$,
there exists a function $f:I\to I$ 
such that:
\begin{itemize}
\item[(i)] $X^f\not\in\mathscr{G}$;
\item[(ii)] $X^f$ is a dense connected subset of $I^2$;
\item[(iii)] $X^f$ is a king space having exactly two continuous 
weak selections. 
\end{itemize}
\end{corollary}
\begin{proof}
Given a family $\mathscr{G}$ satisfying the assumptions of our corollary, let
 $f:I\to I$ be the function satisfying the conclusion of Lemma 
\ref{counter-example:3}. 
Then (i) holds.
By Lemma \ref{counter-example:2}, (ii) holds as well.
Finally, (iii) follows from (ii) and Lemma \ref{connected:king}.
\end{proof}

Recall that a continuous image of the irrational numbers 
is called an {\em analytic\/} space. All complete separable metric spaces 
are analytic; in particular, compact metric spaces are analytic.

\begin{example}
\label{example}
Let $\mathscr{G}$ be the family of all analytic subsets of $I^2$.
Since $\mathscr{G}$ has cardinality $\cont$ (see, for example, \cite{Kechris}), 
we can apply Corollary \ref{analytic:corollary} to this $\mathscr{G}$
to get the function $f:I\to I$ as in the conclusion of this corollary.
Then {\em $X^f$ is a connected, separable metric king space having exactly two continuous 
weak selections such that $X^f$ is not analytic\/}. 
In particular, $X^f$ is not completely metrizable.
\end{example}

It follows from Theorem \ref{reverse:theorem}(iii) that connectedness cannot be replaced 
with
local connectedness in 
Examples  \ref{example:2} and \ref{example}.

\section{Open questions}

\begin{problem}
Find a characterization of Hausdorff (Tychonoff) king spaces.
\end{problem}

Recall that a space is {\em locally pseudocompact\/} if every point of it
 has an open neighbourhood whose closure is pseudocompact.  

\begin{question}
Let $X$ be a Hausdorff (Tychonoff) king space that has at least one continuous 
weak selection. 
Must $X$ be compact if it satisfies either of the conditions below?
\begin{itemize}
\item[(i)] $X$ is locally compact;
\item[(ii)] $X$ is locally pseudocompact.
\end{itemize}
\end{question}

By Theorem \ref{reverse:theorem}, 
the answer 
to this question is positive
when $X$ is additionally 
assumed to be either locally connected or zero-dimensional.
According to Remark \ref{locally:compact:remark}, the answer 
to item (i) of this question is also positive when $X$ is  
connected.

\medskip
\noindent
{\bf Acknowledgements.\/}
We are grateful to Tsugunori Nogura for a stimulating discussion
of the proof of Theorem \ref{kings:in:compact:spaces}.
Our special thanks go to Aarno Hohti for 
useful comments on preliminary versions of the paper and supplying 
the reference \cite{Magill}.
We are obliged to Shimpei Yamauchi for helpful remarks.

\end{document}